\documentclass[11pt]{amsart}
\usepackage[centertags]{amsmath}
\usepackage{amsfonts}
\usepackage{amssymb}
\usepackage{amsthm}
\usepackage{newlfont}
\usepackage{amscd}
\usepackage{mathrsfs}
\usepackage[all,cmtip]{xy}
\usepackage{tikz-cd}
\tikzcdset{arrow style=Latin Modern, row sep/normal=1.35em, column sep/normal=1.8em, cramped}
\usepackage{tikz}
\usepackage[pagebackref=false,colorlinks]{hyperref}
\usepackage{array}

\definecolor{mycolor}{HTML}{F7F8E0}
\definecolor{myorange}{RGB}{245,156,74}
\definecolor{cadetgrey}{rgb}{0.57, 0.64, 0.69}
\definecolor{calpolypomonagreen}{rgb}{0.12, 0.3, 0.17}

\hypersetup{pdffitwindow=true,linkcolor=calpolypomonagreen,citecolor=calpolypomonagreen,urlcolor=calpolypomonagreen}
\usepackage{cite}
\usepackage{fancyhdr}
\usepackage{stmaryrd}

\usepackage[OT2,OT1]{fontenc}
\newcommand\cyr{%
\renewcommand\rmdefault{wncyr}%
\renewcommand\sfdefault{wncyss}%
\renewcommand\encodingdefault{OT2}%
\normalfont
\selectfont}
\DeclareTextFontCommand{\textcyr}{\cyr}

\usepackage[toc, page] {appendix}

\numberwithin{equation}{section}

\newtheorem{thm}{Theorem}[section]

\newtheorem{cor}[thm]{Corollary}
\newtheorem{lem}[thm]{Lemma}
\newtheorem{prop}[thm]{Proposition}

\theoremstyle{definition}

\begin{document}

\title[Fitting ideals of Selmer groups]{On the Fitting ideals of anticyclotomic Selmer groups of elliptic curves with good ordinary reduction}
\author{Chan-Ho Kim}
\address{
Department of Mathematics and Institute of Pure and Applied Mathematics,
Jeonbuk National University,
567 Baekje-daero, Deokjin-gu, Jeonju, Jeollabuk-do 54896, Republic of Korea
}
\email{chanho.math@gmail.com}
\thanks{Chan-Ho Kim was partially supported 
by a KIAS Individual Grant (SP054103) via the Center for Mathematical Challenges at Korea Institute for Advanced Study,
by the National Research Foundation of Korea(NRF) grant funded by the Korea government(MSIT) (No. 2018R1C1B6007009, 2019R1A6A1A11051177), 
by research funds for newly appointed professors of Jeonbuk National University in 2024, and
by Global-Learning \& Academic research institution for Master’s$\cdot$Ph.D. Students, and Postdocs (LAMP) Program of the National Research Foundation of Korea (NRF) funded by the Ministry of Education (No. RS-2024-00443714).
}
\date{\today}
\begin{abstract}
We give a short proof of the anticyclotomic analogue of the ``strong" main conjecture of Kurihara on Fitting ideals of Selmer groups for elliptic curves with good ordinary reduction  under mild hypotheses.
More precisely, we completely determine the initial Fitting ideal of Selmer groups over finite subextensions of an imaginary quadratic field in its anticyclotomic $\mathbb{Z}_p$-extension in terms of Bertolini--Darmon's theta elements.
\end{abstract}
\maketitle

\setcounter{tocdepth}{1}
\section{Introduction}
\subsection{The statement of the main result}
Let $E$ be an elliptic curve of conductor $N$ over $\mathbb{Q}$ and $p \geq 5$ be a prime of good ordinary reduction for $E$ such that
\begin{itemize}
\item[(Im)]  the mod $p$ Galois representation
$\overline{\rho} : G_{\mathbb{Q}} = \mathrm{Gal}(\overline{\mathbb{Q}}/\mathbb{Q}) \to \mathrm{Aut}_{\mathbb{F}_p}(E[p]) \simeq \mathrm{GL}_2(\mathbb{F}_p)$ is surjective, and
\item[(Ram)] $\overline{\rho}$ is ramified at every prime dividing $N$, so $p$ does not divide Tamagawa factors of $E$.
\end{itemize}
Let $K$ be an imaginary quadratic field of odd discriminant $-D_K  < -4$ with $(D_K, Np) = 1$ such that
\begin{itemize}
\item[(Spl)] $p$ splits in $K$, and
\item[(Na)] $a_p(E) \not\equiv 1 \pmod{p}$.
\end{itemize}
Write 
$$N = N^+ \cdot N^-$$
 where a prime divisor of $N^+$ splits in $K$ and a prime divisor of $N^- $ is inert in $K$.
\begin{itemize}
\item[(Def)] Assume that $N^-$ is a square-free product of an odd number of primes.
\end{itemize}
Let $K_\infty$ be the anticyclotomic $\mathbb{Z}_p$-extension of $K$ and $K_n$ be the subextension of $K$ in $K_\infty$ of degree $p^n$ for $n \geq 0$.
Let $\Lambda_n = \mathbb{Z}_p[\mathrm{Gal}(K_n/K)] \simeq \mathbb{Z}_p[X]/( (1+X)^{p^n} - 1 )$ be the finite layer Iwasawa algebra.
Under (Def), denote by
$$\theta(E/K_n) = \sum_{\sigma \in \mathrm{Gal}(K_n/K)} a_{\sigma} \cdot \sigma \in \Lambda_n$$
 Bertolini--Darmon's theta element of $E$ over $K_n$ which interpolates the square-roots of $L(E/K, \chi, 1)$ for finite order characters $\chi$ on $\mathrm{Gal}(K_n/K)$. It is reviewed in $\S$\ref{sec:bertolini-darmon-elements}.
For the natural projection map $\pi_{n, n-1} : \Lambda_{n} \to \Lambda_{n-1}$, let 
$\nu_{n-1, n} : \Lambda_{n-1} \to \Lambda_n$ be the map defined by $\sigma \mapsto \sum_{\pi_{n, n-1}(\tau) = \sigma} \tau$. For $0 \leq m \leq n$, we write
 $\nu_{m, n} = \nu_{n-1, n} \circ \nu_{n-2, n-1} \circ \cdots \circ \nu_{m, m+1} $.
Then we have the equality of ideals of $\Lambda_n$ (Lemma \ref{lem:reduction-generators})
 $$\left( \theta(E/K_n) , \nu_{n-1, n} \left( \theta(E/K_{n-1}) \right) \right) = \left( \nu_{m, n} \left( \theta(E/K_{m}) \right) : 0 \leq m \leq n \right) ,$$
and it is a principal ideal under (Na) (Lemma \ref{lem:p-stabilization}).
  
The goal of this article is to prove the following anticyclotomic analogue of the ``strong" main conjecture of Kurihara \cite[Conj. 0.3]{kurihara-invent}, which refines the ``weak" main conjecture of Mazur and Tate \cite[Conj. 3]{mazur-tate}.
\begin{thm} \label{thm:main}
Under the assumptions mentioned above, i.e. (Im),(Ram),(Spl),(Na), and (Def), 
the theta elements over $K_m$ with $0 \leq m \leq n$ generate the initial Fitting ideal of dual Selmer groups over $K_n$, i.e.
we have equality of ideals of $\Lambda_n$
$$\left( \theta(E/K_n) , \nu_{n-1, n} \left( \theta(E/K_{n-1}) \right) \right)^2  = \mathrm{Fitt}_{\Lambda_n} \left( \mathrm{Sel}(K_n, E[p^\infty])^\vee \right) ,$$
which is indeed a principal ideal, where $\mathrm{Sel}(K_n, E[p^\infty])$ is the classical Selmer group of $E[p^\infty]$ over $K_n$ and $(-)^\vee$ means the Pontryagin dual.
\end{thm}
Since theta elements interpolate the \emph{square-roots} of twisted Rankin--Selberg $L$-values, it is natural that the \emph{square} of the ideal generated by theta elements appears in the equality. 

The strategy of our proof follows that given in \cite{kim-kurihara} and we also add some details on the ``$p$-destabilization" process and on the comparison of various anticyclotomic Selmer groups of elliptic curves.

\subsection*{Acknowledgement}
I would like to thank Rob Pollack and Masato Kurihara very much.
Rob Pollack first suggested me to think about this problem when I was a PhD student and I was able to solve it after having the discussions with Masato Kurihara.
In particular, I learned the idea of the 
proof of Lemma \ref{lem:p-stabilization} on the $p$-destabilization process from the discussion with Masato Kurihara at his home, January 3, 2017.
The main result of this article partially refines that of \cite{kim-mazur-tate}.
We are grateful to thank the referee for pointing out inaccuracies and unrefined writings.
The explanation is greatly improved in the revised version following the referee's suggestions.

\section{Bertolini--Darmon's theta elements and anticyclotomic $p$-adic $L$-functions} \label{sec:bertolini-darmon-elements}
We quickly review the construction of Gross points of conductor $p^n$, theta elements, and anticyclotomic $p$-adic $L$-functions. 
See  \cite{chida-hsieh-main-conj, chida-hsieh-p-adic-L-functions, kim-overconvergent} for details.
\subsection{Gross points}
Let $K$ be the imaginary quadratic field of odd discriminant $-D_K < - 4$.
Define
$$\vartheta :=      \dfrac{ D_K - \sqrt{-D_K} }{ 2 }
$$
 so that
$\mathcal{O}_K = \mathbb{Z} + \mathbb{Z}\vartheta$.
Let $B_{N^-}$ be the definite quaternion algebra over $\mathbb{Q}$ of discriminant $N^-$.
Then there exists an embedding $\Psi :K \hookrightarrow B_{N^-}$ \cite{vigneras}.
More explicitly, we choose a $K$-basis $(1,J)$ of $B_{N^-}$ so that $B_{N^-} = K \oplus K \cdot J$ such that
$\beta := J^2 \in \mathbb{Q}^\times$ with $\beta <0$, $J \cdot t = \overline{t} \cdot J$ for all $t \in K$, $\beta \in \left( \mathbb{Z}^\times_q \right)^2$ for all $q \mid pN^+$, and
$\beta \in \mathbb{Z}^\times_q$ for all $q \mid D_K$.
Fix a square root $\sqrt{\beta} \in \overline{\mathbb{Q}}$ of $\beta$.
For a $\mathbb{Z}$-module $A$, write $\widehat{A} = A \otimes \widehat{\mathbb{Z}}$.
Fix an isomorphism
$$i := \prod i_q : \widehat{B}^{(N^-)}_{N^-} \simeq \mathrm{M}_2(\mathbb{A}^{(N^-\infty)})$$
as follows:
\begin{itemize}
\item For each finite place $q \mid N^+p$, the isomorphism
$i_q : B_{N^-,q} \simeq \mathrm{M}_2(\mathbb{Q}_q)$ is defined by
\[
\xymatrix{
{
i_q(\vartheta)  = \left( \begin{matrix}
\mathrm{trd}(\vartheta) & - \mathrm{nrd}(\vartheta) \\
1 & 0
\end{matrix} \right)  }, & 
{
i_q(J)  = \sqrt{\beta} \cdot \left( \begin{matrix}
-1 & \mathrm{trd}(\vartheta) \\
0 & 1
\end{matrix} \right) 
}
}
\]
where $\mathrm{trd}$ and $\mathrm{nrd}$ are the reduced trace and the reduced norm on $B$, respectively.
\item For each finite place $q \nmid pN^+$, the isomorphism
$i_q : B_{N^-,q} \simeq \mathrm{M}_2(\mathbb{Q}_q)$ is chosen so that
$i_q \left( \mathcal{O}_K \otimes \mathbb{Z}_q  \right) \subseteq \mathrm{M}_2(\mathbb{Z}_q) $.
\end{itemize}
Under the fixed isomorphism $i$, for any rational prime $q$, the local Gross point $\varsigma_q \in B^\times_{N^-,q}$ is defined as follows:
\begin{itemize}
\item $\varsigma_q := 1$
in $B^\times_{N^-,q}$ for $q \nmid pN^+$.
\item $\varsigma_q := \frac{1}{\sqrt{D_K}}\cdot \left( \begin{matrix}
\vartheta & \overline{\vartheta} \\
1 & 1
\end{matrix} \right) \in \mathrm{GL}_2(K_\mathfrak{q}) = \mathrm{GL}_2(\mathbb{Q}_q) $
for $q \mid N^+$ with $q = \mathfrak{q} \overline{\mathfrak{q}}$ in $\mathcal{O}_K$.
\item 
$\varsigma^{(n)}_p = \left( \begin{matrix}
\vartheta & -1 \\
1 & 0
\end{matrix} \right)
\cdot \left( \begin{matrix}
p^n & 0 \\
0 & 1
\end{matrix} \right)
 \in \mathrm{GL}_2(K_\mathfrak{p}) = \mathrm{GL}_2( \mathbb{Q}_{p} )$ where $p = \mathfrak{p}\overline{\mathfrak{p}}$ splits in $K$.
\end{itemize}
Let $\widehat{\Psi} : \widehat{K} \hookrightarrow\widehat{B}_{N^-}$ be the adelic version of $\Psi$.
We define
 $x_n : \widehat{K}^\times \to \widehat{B}^\times_{N^-}$
by
$x_n(a) = \widehat{\Psi}(a) \cdot \varsigma^{(n)} := \widehat{\Psi}(a) \cdot \left( \varsigma^{(n)}_p \times \prod_{q \neq p} \varsigma_q \right)$.
The collection $\left\lbrace x_n(a) : a \in \widehat{K}^\times \right\rbrace$ of points is called the \textbf{Gross points of conductor $p^n$ on $\widehat{B}^\times_{N^-}$}.
The fixed embedding $K \hookrightarrow B_{N^-}$ also induces an optimal embedding of $\mathcal{O}_n = \mathbb{Z}+p^n \mathcal{O}_K$ into the Eichler order
$B_{N^-} \cap \varsigma^{(n)}\widehat{R}_{N^+}(\varsigma^{(n)})^{-1}$ where $R_{N^+}$ is the Eichler order of level $N^+$ under the fixed isomorphism $i$.

\subsection{Theta elements}
Let $f(z) = \sum_{n \geq 1} a_n q^n \in S_{2}(\Gamma_0(N))$ be the cuspidal newform of weight two with rational Fourier coefficients corresponding to $E$ via the modularity theorem \cite{bcdt}. 
Let $\phi_f : B^\times_{N^-} \backslash \widehat{B}^{\times}_{N^-} / \widehat{R}^\times_{N^+} \to \mathbb{C}$ be the Jacquet--Langlands transfer of $f$.
Since $B^\times_{N^-} \backslash \widehat{B}^{\times}_{N^-} / \widehat{R}^\times_{N^+}$ is a finite set and $f$ is a Hecke eigenform, we are able to and do normalize
$$\phi_f : B^\times_{N^-} \backslash \widehat{B}^{\times}_{N^-} / \widehat{R}^\times_{N^+} \to \mathbb{Z}_p$$
such that the image of $\phi_f $ does not lie in $p\mathbb{Z}_p$. This integral normalization is related to the congruence ideals \cite{pw-mu, kim-summary,kim-ota}.
Let
$$\widetilde{\theta}_n(E/K) = \sum_{[a] \in \mathcal{G}_n} \phi_f (x_n(a) )  \cdot [a] \in \mathbb{Z}_p[\mathcal{G}_n]$$
where $\mathcal{G}_n = K^\times \backslash \widehat{K}^\times / \widehat{\mathcal{O}}^\times_n $ and $[a]$ is the image of $a \in \widehat{K}^\times$ in $\mathcal{G}_n$. Then
\textbf{Bertolini--Darmon's theta element $\theta(E/K_n)$ of $E$ over $K_n$} is defined by the image of $\widetilde{\theta}_n(E/K)$ in $\Lambda_n$
\[
\xymatrix@R=0em{
\mathbb{Z}_p[\mathcal{G}_n] \ar[r] & \Lambda_n = \mathbb{Z}_p[\mathrm{Gal}(K_n/K)]\\
\widetilde{\theta}_n(E/K) \ar@{|->}[r] & \theta(E/K_n) 
}
\]
where the map is naturally induced from the quotient map $\mathcal{G}_n \to \mathrm{Gal}(K_n/K)$.
It is known that $\theta(E/K_n)$ interpolates ``an half of" $L(E, \chi, 1)$ where $ \chi$ runs over characters on $\mathrm{Gal}(K_n/K)$. See \cite[Rem. (iii) after Thm. A]{chida-hsieh-p-adic-L-functions} for the precise meaning of ``an half of".
Because $\theta(E/K_n)$ depends on the choice of Gross points, 
 $\theta(E/K_n)$ is well-defined only up to multiplication by $\mathrm{Gal}(K_n/K)$.
\subsection{$p$-adic $L$-functions}
Let $\alpha , \beta$ be the roots of the Hecke polynomial $X^2 - a_pX + p$ of $f$ at $p$. 
Since $f$ is ordinary at $p$, one of them, say $\alpha$, is a $p$-adic unit.
 
The \textbf{$p$-stabilization $f_\alpha \in S_{2}(\Gamma_0(Np))$ of $f$} is defined by
$$f_\alpha (z) = f(z) - \beta \cdot f(pz) $$
whose $U_p$-eigenvalue is $\alpha$.
Then the \textbf{theta element of $f_\alpha$ over $K_n$} is characterized by the following relation:
\begin{equation} \label{eqn:p-stabilization}
\theta(f_\alpha/K_n) = \dfrac{1}{\alpha^n} \cdot \left(  \theta(E/K_n) - \dfrac{1}{\alpha} \cdot \nu_{n-1, n}( \theta(E/K_{n-1}) )  \right) .
\end{equation}
It is known that theta elements of $E$ satisfies the three term relation (e.g. \cite[Lem. 2.6]{darmon-iovita})
\begin{equation} \label{eqn:three-term-relation}
\pi_{n+1, n} \left( \theta(E/K_{n+1}) \right) = a_p \cdot \theta(E/K_n) - \nu_{n-1, n} \left( \theta(E/K_{n-1}) \right)
\end{equation}
and the theta elements of $f_\alpha$ satisfy the norm compatibility
\begin{equation} \label{eqn:norm-compatibility}
\pi_{n+1, n} \left( \theta(f_\alpha/K_{n+1}) \right)=  \theta(f_\alpha/K_{n}) .
\end{equation}
Let $\iota$ be the involution on $\Lambda_n$ defined by inverting group-like elements, so we have
 $$\iota( \sum_{\sigma \in \mathrm{Gal}(K_n/K)} a_{\sigma} \cdot \sigma ) = \sum_{\sigma \in \mathrm{Gal}(K_n/K)} a_{\sigma} \cdot \sigma^{-1}.$$
We define the \textbf{anticyclotomic $p$-adic $L$-function of $E$} by
$$L_p(E/K_\infty) =   \varprojlim_n\left(  \theta(f_\alpha/K_n) \cdot \iota( \theta(f_\alpha/K_n) ) \right) \in \Lambda = \varprojlim_n \Lambda_n.$$
This element is well-defined.
The functional equation for Bertolini--Darmon's theta elements yields the equality of ideals of $\Lambda$
(e.g. \cite[Prop. 2.13]{bertolini-darmon-mumford-tate-1996}, \cite[Lem. 1.5]{bertolini-darmon-imc-2005})
\begin{equation} \label{eqn:f-e}
\left( \theta(f_\alpha/K_n) \right) = \left( \iota( \theta(f_\alpha/K_n) ) \right).
\end{equation}
We prove two useful lemmas.
\begin{lem} \label{lem:reduction-generators}
We have an equality of ideals of $\Lambda_n$
$$\left( \theta(E/K_n) , \nu_{n-1, n} \left( \theta(E/K_{n-1}) \right) \right) = \left( \nu_{m, n} \left( \theta(E/K_{m}) \right) : 0 \leq m \leq n \right) .$$
\end{lem}
\begin{proof}
From the three term relation (\ref{eqn:three-term-relation}), we have
$$\nu_{n-1, n} \left( \pi_{n, n-1} \left( \theta(E/K_{n}) \right) \right) = a_p \cdot \nu_{n-1, n}  \left( \theta(E/K_{n-1}) \right) - \nu_{n-2, n} \left( \theta(E/K_{n-2}) \right) $$
for $n \geq 2$.
Since 
$\nu_{n-1, n} \left( \pi_{n, n-1} \left( \theta(E/K_{n}) \right) \right) = f_n \cdot \theta(E/K_{n})$
for some $f_n \in \Lambda_n$, we have
$$\nu_{n-2, n} \left( \theta(E/K_{n-2}) \right) \in \left( \theta(E/K_n) , \nu_{n-1, n} \left( \theta(E/K_{n-1}) \right) \right) \subseteq \Lambda_n.$$
In the same manner, we can obtain
$$\nu_{n-3, n-1} \left( \theta(E/K_{n-3}) \right) \in \left( \theta(E/K_{n-1}) , \nu_{n-2, n-1} \left( \theta(E/K_{n-2}) \right) \right) \subseteq \Lambda_{n-1}. $$
By taking $\nu_{n-1, n}$, we have
\begin{align*}
\nu_{n-3, n} \left( \theta(E/K_{n-3}) \right) & \in \left( \nu_{n-1, n} \left( \theta(E/K_{n-1}) \right) , \nu_{n-2, n} \left( \theta(E/K_{n-2}) \right) \right) \\
& \subseteq \left( \theta(E/K_{n}), \nu_{n-1, n} \left( \theta(E/K_{n-1}) \right)  \right) \\
& \subseteq \Lambda_{n}. 
\end{align*}
By applying this argument recursively, the conclusion follows.
\end{proof}

\begin{lem} \label{lem:p-stabilization}
Under (Spl) and (Na), we have an equality of ideals of $\Lambda_n$
 $$\left( \theta(E/K_n) , \nu_{n-1, n} \left( \theta(E/K_{n-1}) \right) \right) = \left(  \theta(f_\alpha/K_{n}) \right)  .$$
\end{lem}
\begin{proof}
By the definition of the $p$-stabilization (\ref{eqn:p-stabilization}), we have one inclusion $\supseteq$.
Hence, we focus on the opposite inclusion.
By the interpolation formula of the anticyclotomic $p$-adic $L$-functions \cite[Thm. A]{chida-hsieh-p-adic-L-functions} under (Spl), we have the comparison of (the square-roots of) $L$-values
$$\theta(f_\alpha/K) = \left( 1 - \dfrac{1}{\alpha} \right) \cdot \theta(E/K).$$
Under (Na), we have equality in $\mathbb{Z}_p$
$$\left( 1 - \dfrac{1}{\alpha} \right)^{-1} \cdot \theta(f_\alpha/K) = \theta(E/K),$$
so we have  $(\theta(E/K)) \subseteq (\theta(f_\alpha/K))$. In fact, they are the same ideal.
From (\ref{eqn:p-stabilization}) and (\ref{eqn:norm-compatibility}), we have
\begin{align*}
\theta(f_\alpha/K_1) & = \dfrac{1}{\alpha} \cdot \left(  \theta(E/K_1) - \dfrac{1}{\alpha} \cdot \nu_{0, 1}( \theta(E/K) )  \right) \\
& =  \dfrac{1}{\alpha} \cdot \left(  \theta(E/K_1) - \dfrac{1}{\alpha} \cdot \nu_{0, 1}( \left( 1 - \dfrac{1}{\alpha} \right)^{-1} \cdot \theta(f_\alpha/K) )  \right) \\
& = \dfrac{1}{\alpha} \cdot \left(  \theta(E/K_1) - \dfrac{1}{\alpha} \cdot \left( 1 - \dfrac{1}{\alpha} \right)^{-1} \cdot \nu_{0, 1}(  \pi_{1,0} ( \theta(f_\alpha/K_1)) )  \right) \\
& = \dfrac{1}{\alpha} \cdot \left(  \theta(E/K_1) - \dfrac{1}{\alpha} \cdot \left( 1 - \dfrac{1}{\alpha} \right)^{-1} \cdot f_1 \cdot \theta(f_\alpha/K_1)   \right)
\end{align*}
for some $f_1 \in \Lambda_1$. This shows that
$\theta(E/K_{1}) = g_{1} \cdot \theta(f_\alpha/K_{1})$
for some $g_{1} \in \Lambda_{1}$.

We suppose that
$\theta(E/K_{n-1}) = g_{n-1} \cdot \theta(f_\alpha/K_{n-1})$
for some $g_{n-1} \in \Lambda_{n-1}$.
\begin{align*}
\theta(f_\alpha/K_n) & = \dfrac{1}{\alpha^n} \cdot \left(  \theta(E/K_n) - \dfrac{1}{\alpha} \cdot \nu_{n-1, n}( \theta(E/K_{n-1}) )  \right) \\
& =  \dfrac{1}{\alpha^n} \cdot \left(  \theta(E/K_{n}) - \dfrac{1}{\alpha} \cdot \nu_{n-1, n}( g_{n-1} \cdot \theta(f_\alpha/K_{n-1}) )  \right) \\
& = \dfrac{1}{\alpha^n} \cdot \left(  \theta(E/K_n) - \dfrac{1}{\alpha} \cdot g_{n-1} \cdot \nu_{n-1, n}(  \pi_{n,n-1} ( \theta(f_\alpha/K_n)) )  \right) \\
& = \dfrac{1}{\alpha^n} \cdot \left(  \theta(E/K_1) - \dfrac{1}{\alpha} \cdot g_{n-1} \cdot f_{n} \cdot \theta(f_\alpha/K_n)   \right)
\end{align*}
for some $f_n \in \Lambda_n$.
 This shows that
$\theta(E/K_{n}) = g_{n} \cdot \theta(f_\alpha/K_{n})$
for some $g_{n} \in \Lambda_{n}$.
By induction, we have inclusion
$$(\theta(E/K_n)) \subseteq (\theta(f_\alpha/K_n)),$$
so we also have
$$( \nu_{n-1, n} \left( \theta(E/K_{n-1}) \right) ) \subseteq ( \nu_{n-1, n} \left( \theta(f_\alpha/K_{n-1}) \right)).$$
Since $\nu_{n-1, n} \left( \theta(f_\alpha/K_{n-1}) \right) = f_{n} \cdot \theta(f_\alpha/K_n)$, we have
$$( \nu_{n-1, n} \left( \theta(f_\alpha/K_{n-1}) \right)) \subseteq ( \theta(f_\alpha/K_{n}) ).$$
The conclusion follows.
\end{proof}

\section{Comparison of Selmer groups}
\subsection{Local properties of Galois representations}
Let $\rho : G_{\mathbb{Q}} \to \mathrm{Aut}_{\mathbb{Q}_p}(V) = \mathrm{GL}_2(\mathbb{Q}_p)$ be the two-dimensional Galois representation associated to $E$.
\begin{itemize}
\item Since $E$ is good ordinary at $p$, we have
$$\rho \vert_{ G_{\mathbb{Q}_p} } \sim 
\begin{pmatrix}
\chi^{-1}_{\alpha} \cdot \chi_{\mathrm{cyc}} & * \\
0  & \chi_{\alpha}
\end{pmatrix}$$
where $\chi_{\alpha}$ is the unramified character sending the arithmetic Frobenius at $p$ to $\alpha$.
\item For $\ell$ dividing $N$ exactly, we also have
$$\rho \vert_{ G_{\mathbb{Q}_\ell} } \sim 
\begin{pmatrix}
\pm \chi_{\mathrm{cyc}} & * \\
0  & \pm \mathbf{1}
\end{pmatrix} .$$
\end{itemize}
For a rational prime $v$ dividing $N^-p$, we consider the following subspaces:
\begin{itemize}
\item For $v= p$, let $F^+V \subseteq V$ be the subspace on which the inertia subgroup $I_v$ acts by $\chi_{\mathrm{cyc}}$.
\item For a rational prime $v$ dividing $N^-$, $F^+V \subseteq V$ be the subspace on which the inertia subgroup $I_v$ acts by $\chi_{\mathrm{cyc}}$ or $\chi_{\mathrm{cyc}} \tau_v$
where $\tau_v$ is the non-trivial unramified quadratic character of $G_{\mathbb{Q}_v}$.
\end{itemize}
Let $L$ be an algebraic extension of $K$.
For a prime $w$ of $L$ dividing $N^-p$, we define the \textbf{ordinary local condition of $V$ at $w$} by
$$\mathrm{H}^1_{\mathrm{ord}}(L_w, V) = \mathrm{ker} \left( \mathrm{H}^1(L_w, V) \to \mathrm{H}^1(L_w, V/F^+V) \right).$$
Denote by $T = \varprojlim_k E[p^k]$ the $p$-adic Tate module of $E$, so we have $T \otimes_{\mathbb{Z}_p} \mathbb{Q}_p = V$, and by $E[p^\infty] = \varinjlim_k E[p^k]$ the $p$-power torsion points of $E$.
Then the same local conditions for $T$, $T/p^kT$, $E[p^\infty]$, and $E[p^k]$ are defined by propagation.

\subsection{$N^-$-ordinary (residual) Selmer groups}
Let $\Sigma$ be the finite set of places of $\mathbb{Q}$ consisting of the places dividing $Np\infty$, and $K_\Sigma$ be the maximal extension of $K$ unramified outside $\Sigma$.
We write 
\begin{itemize}
\item $\Sigma^+ \subseteq \Sigma$
to be the subset of $\Sigma$ consisting of the places not dividing $p\infty$ which split in $K/\mathbb{Q}$, and
\item $\Sigma^- \subseteq \Sigma$
to be the subset of $\Sigma$ consisting of the places not dividing $p\infty$ which are inert in $K/\mathbb{Q}$.
\end{itemize}
For a place $w$ of $K_\infty$, we write $w \in \Sigma^{\pm}$ if $w$ divides a rational prime $\ell$ contained in $\Sigma^{\pm}$, respectively.
For every $k \geq 1$, we define the \textbf{$N^-$-ordinary (and $N^+$-strict) Selmer group of $E[p^k]$} 
$\mathrm{Sel}_{N^-}(K_\infty, E[p^k])$
by the kernel of the map
$$\mathrm{H}^1(K_\Sigma/ K_\infty, E[p^k]) \to \prod_{w \nmid \Sigma^+} \mathrm{H}^1(K_{\infty, w}, E[p^k]) \times \prod_{w \in \Sigma^- \textrm{ or } w \vert p} \dfrac{\mathrm{H}^1(K_{\infty,w}, E[p^k])}{\mathrm{H}^1_{\mathrm{ord}}(K_{\infty,w}, E[p^k])}$$ 
and define
$\mathrm{Sel}_{N^-}(K_\infty, E[p^\infty]) = \varinjlim_k \mathrm{Sel}_{N^-}(K_\infty, E[p^k])$.
This is \emph{the} Selmer group used in the bipartite Euler system argument \cite[Def. 2.8]{bertolini-darmon-imc-2005}.

\subsection{Minimal and Greenberg Selmer groups}
We follow the convention of \cite[\S3.1]{pw-mu}.
The \textbf{minimal Selmer group $\mathrm{Sel}_{\mathrm{min}}(K_\infty, E[p^\infty])$ of $E[p^\infty]$} is defined by the kernel of the map
$$\mathrm{H}^1(K_\infty, E[p^\infty]) \to \prod_{w \nmid p} \mathrm{H}^1(K_{\infty, w}, E[p^\infty]) \times \prod_{w \vert p} \dfrac{\mathrm{H}^1(K_{\infty,w}, E[p^\infty])}{\mathrm{H}^1_{\mathrm{ord}}(K_{\infty,w}, E[p^\infty])},$$
and the \textbf{Greenberg Selmer group $\mathrm{Sel}_{\mathrm{Gr}}(K_\infty, E[p^\infty])$ of $E[p^\infty]$} is defined by the kernel of the map
$$\mathrm{H}^1(K_\infty, E[p^\infty]) \to \prod_{w \nmid p} \mathrm{H}^1(I_{\infty, w}, E[p^\infty]) \times 
\prod_{w \vert p} \dfrac{\mathrm{H}^1(K_{\infty,w}, E[p^\infty])}{\mathrm{H}^1_{\mathrm{ord}}(K_{\infty,w}, E[p^\infty])}$$
where $I_{\infty, w}$ is the inertia subgroup of $G_{K_{\infty,w}}$.
Under (Ram), $\overline{\rho}$ is ramified at every prime dividing $N$, so $p$ does not divide any Tamagawa factors.
Then by using \cite[Lem. 3.4]{pw-mu}, we have an isomorphism
\begin{equation} \label{eqn:minimal-greenberg}
\mathrm{Sel}_{\mathrm{min}}(K_\infty, E[p^\infty]) \simeq \mathrm{Sel}_{\mathrm{Gr}}(K_\infty, E[p^\infty]).
\end{equation}
\subsection{The comparison}
We recall the final displayed equation in the proof of \cite[Prop. 3.6]{pw-mu}:
\[
\xymatrix{
0 \ar[r] &\mathrm{Sel}_{N^-}(K_\infty, E[p^k]) \ar[r] & \mathrm{Sel}_{\mathrm{min}}(K_\infty, E[p^\infty])[p^k] \ar[r] &  \prod_{w} 
\dfrac{(E[p^\infty])^{G_{K_{\infty, w}}}}{p^k (E[p^\infty])^{G_{K_{\infty, w}}}} 
}
\]
where $w$ runs over the primes of $K_\infty$ dividing $N^+$. The local conditions at primes dividing $N^-$ of minimal Selmer groups and $N^-$-ordinary Selmer groups coincide since such primes split completely in $K_\infty/K$.
Thus, we have inclusion
$$\mathrm{Sel}_{N^-}(K_\infty, E[p^k]) \subseteq \mathrm{Sel}_{\mathrm{min}}(K_\infty, E[p^\infty])[p^k]$$
which is of finite index and is independent of $k$.
\begin{prop} \label{prop:cotorsion-implies-cotorsion}
If $\mathrm{Sel}_{N^-}(K_\infty, E[p^\infty])$ is $\Lambda$-cotorsion with vanishing of $\mu$-invariant, then
 $\mathrm{Sel}_{\mathrm{min}}(K_\infty, E[p^\infty])$ is also $\Lambda$-cotorsion with vanishing of $\mu$-invariant.
\end{prop} 
\begin{proof}
We have
$\mathrm{Sel}_{N^-}(K_\infty, E[p^\infty])[p] = \mathrm{Sel}_{N^-}(K_\infty, E[p])$
since $N^-$-ordinary Selmer groups of $E[p^\infty]$ are defined as the injective limit of $N^-$-ordinary Selmer groups of $E[p^k]$.
By the assumption,
$\mathrm{Sel}_{N^-}(K_\infty, E[p])$ is finite as noted in the proof of \cite[Cor. 2.3]{kim-pollack-weston}.
Since the inclusion 
$\mathrm{Sel}_{N^-}(K_\infty, E[p]) \subseteq \mathrm{Sel}_{\mathrm{min}}(K_\infty, E[p^\infty])[p]$
is of finite index, 
$\mathrm{Sel}_{\mathrm{min}}(K_\infty, E[p^\infty])[p]$ is also finite. 
By the same reasoning, the conclusion follows.
\end{proof}
\begin{prop} \label{prop:no-finite}
Under (Im), if $\mathrm{Sel}_{\mathrm{min}}(K_\infty, E[p^\infty])$ is $\Lambda$-cotorsion, then
$\mathrm{Sel}_{\mathrm{min}}(K_\infty, E[p^\infty])$ has no proper $\Lambda$-submodule of finite index.
Thus, we have
\begin{align*}
\mathrm{char}_{\Lambda} \mathrm{Sel}_{\mathrm{min}}(K_\infty, E[p^\infty]) & = \mathrm{Fitt}_{\Lambda} \mathrm{Sel}_{\mathrm{min}}(K_\infty, E[p^\infty]) ,\\
\mathrm{Sel}_{\mathrm{min}}(K_\infty, E[p^\infty]) & \simeq \mathrm{Sel}_{N^-}(K_\infty, E[p^\infty]) .
\end{align*}
\end{prop}
\begin{proof}
This follows from \cite[Prop. 4.14]{greenberg-lnm}, which covers the cyclotomic case actually, but the argument generalizes to our setting as mentioned in the proof of \cite[Prop. 3.6]{pw-mu}.
\end{proof}
The following corollary follows from (\ref{eqn:minimal-greenberg}) and the above two propositions.
\begin{cor} \label{cor:selmer}
Under (Im) and (Ram), if $\mathrm{Sel}_{N^-}(K_\infty, E[p^\infty])$ is $\Lambda$-cotorsion with vanishing of $\mu$-invariant, we have isomorphisms
$$\mathrm{Sel}_{N^-}(K_\infty, E[p^\infty]) \simeq \mathrm{Sel}_{\mathrm{min}}(K_\infty, E[p^\infty]) \simeq \mathrm{Sel}_{\mathrm{Gr}}(K_\infty, E[p^\infty]).$$
\end{cor}

\section{The proof of the main theorem via Iwasawa theory} \label{sec:iwasawa-theory}
We first gather some tools from Iwasawa theory and give a proof of Theorem \ref{thm:main}.
\subsection{Iwasawa theory}
The anticyclotomic main conjecture for $(E,p, K)$ is now completely known for our setting. 
\begin{thm} \label{thm:main-conj}
Under (Im),(Ram),(Spl),(Na), and (Def), we have the following statements:
\begin{enumerate}
\item $L_p(E/K_\infty)$ is non-zero.
\item $\mu(L_p(E/K_\infty))= 0 $.
\item $\mathrm{Sel}_{N^-}(K_\infty, E[p^\infty])$ is $\Lambda$-cotorsion with vanishing of $\mu$-invariants.
\item $\left( L_p(E/K_\infty) \right) = \mathrm{char}_\Lambda \left( \mathrm{Sel}(K_\infty, E[p^\infty])^\vee \right)$.
\end{enumerate}
\end{thm}
\begin{proof}
\begin{enumerate}
\item It is proved in \cite{vatsal-uniform}.
\item It is proved in \cite{vatsal-duke}.
\item This follows from (1), (2), and the Euler system divisibility
$$\left( L_p(E/K_\infty) \right) \subseteq \mathrm{char}_\Lambda \left( \mathrm{Sel}_{N^-}(K_\infty, E[p^\infty])^\vee \right)$$
obtained from the bipartite Euler system argument \cite{bertolini-darmon-imc-2005, pw-mu}.
Condition (Na) is implicitly used in the Euler system argument. See \cite[Assu. 1.1 and Rem. 1.4]{kim-pollack-weston} for this issue. 
\item By using (3) and Corollary \ref{cor:selmer}, we can identify $\mathrm{Sel}_{N^-}(K_\infty, E[p^\infty])$ with the minimal Selmer group. By \cite[Prop. 2.1]{greenberg-lnm}, the minimal Selmer group also coincides with the classical Selmer group.
The opposite divisibility 
$$\left( L_p(E/K_\infty) \right) \supseteq \mathrm{char}_\Lambda \left( \mathrm{Sel}(K_\infty, E[p^\infty])^\vee \right)$$
follows from \cite{skinner-urban}.
Condition (Spl) is needed only for this last statement to invoke \cite{skinner-urban}.
\end{enumerate}
\end{proof}
\begin{cor} \label{cor:no-finite}
Under (Im),(Ram),(Na), and (Def), the classical Selmer group $\mathrm{Sel}(K_\infty, E[p^\infty])$ has no proper $\Lambda$-submodule of finite index; thus, 
$$\mathrm{char}_{\Lambda} \mathrm{Sel}(K_\infty, E[p^\infty]) = \mathrm{Fitt}_{\Lambda} \mathrm{Sel}(K_\infty, E[p^\infty]) .$$
\end{cor}
\begin{proof}
By Theorem \ref{thm:main-conj}.(3), $\mathrm{Sel}_{N^-}(K_\infty, E[p^\infty])$ is $\Lambda$-cotorsion.
By Proposition \ref{prop:cotorsion-implies-cotorsion},  $\mathrm{Sel}_{\mathrm{min}}(K_\infty, E[p^\infty])$ is also $\Lambda$-cotorsion.
The conclusion follows from Proposition \ref{prop:no-finite} and the identification of the minimal Selmer group and the classical Selmer group.
\end{proof}
We recall the control theorem.
\begin{prop}[Control theorem] \label{prop:control}
Let $\omega_n = \omega_n(X) = (1+X)^{p^n} -1 \in \mathbb{Z}_p\llbracket X \rrbracket \simeq \Lambda$.  
Under (Im), (Na), and (Def), the restriction map
$$\mathrm{Sel}(K_n, E[p^\infty]) \to \mathrm{Sel}(K_\infty, E[p^\infty])[\omega_n] $$
is injective with the finite cokernel whose size is bounded independently of $n$. If we further assume (Ram),
then it is an isomorphism.
\end{prop}
\begin{proof}
See \cite[Prop. 1.9]{chida-hsieh-main-conj} with the identifications of Selmer groups in Corollaries \ref{cor:selmer} and \ref{cor:no-finite}.
\end{proof}

\subsection{The proof of Theorem \ref{thm:main}}
By the anticyclotomic main conjecture (Theorem \ref{thm:main-conj}), we have
$$\left( L_p(E/K_\infty) \right) = \mathrm{char}_{\Lambda} \left( \mathrm{Sel}(K_\infty, E[p^\infty])^\vee \right) .$$
By Corollary \ref{cor:no-finite}, the above equality becomes
$$\left( L_p(E/K_\infty) \right) = \mathrm{Fitt}_{\Lambda} \left( \mathrm{Sel}(K_\infty, E[p^\infty])^\vee \right)$$
Under the quotient map $\Lambda \to \Lambda_n = \Lambda / \omega_n$, it becomes
$$\left( \left(  \theta(f_\alpha/K_n) \cdot \iota( \theta(f_\alpha/K_n) ) \right) \right) = \mathrm{Fitt}_{\Lambda_n} \left( \left( \mathrm{Sel}(K_\infty, E[p^\infty]) [\omega_n]\right)^\vee \right)$$
since Fitting ideals are compatible with base change.
By using the functional equation of theta elements (\ref{eqn:f-e}) and the control theorem (Proposition \ref{prop:control}), we have
$$\left(  \theta(f_\alpha/K_n)  \right)^2 = \mathrm{Fitt}_{\Lambda_n} \left(  \mathrm{Sel}(K_n, E[p^\infty])^\vee \right) .$$
Theorem \ref{thm:main} now follows from Lemma \ref{lem:p-stabilization}, and the ideal is principal thanks to the above equality.

\bibliographystyle{amsalpha}
\bibliography{library}

\end{document}